\documentclass[12pt]{amsart}

\usepackage{geometry}
\usepackage{amsthm}
\usepackage{amstext}
\usepackage{amssymb, url}

\usepackage{graphicx}

\newcommand{\hookdownarrow}{\mathrel{\rotatebox[origin=c]{-90}{$\hookrightarrow$}}}

\makeatletter
\numberwithin{equation}{section}
\numberwithin{figure}{section}
\theoremstyle{plain}
\newtheorem{thm}{Theorem}[section]
  \theoremstyle{definition}
  \newtheorem{defn}[thm]{Definition}
  \theoremstyle{remark}
  \newtheorem{rem}[thm]{Remark}
  \theoremstyle{plain}
  \newtheorem{lem}[thm]{Lemma}
  \theoremstyle{plain}
  \newtheorem{prop}[thm]{Proposition}
  \theoremstyle{plain}
  \newtheorem{cor}[thm]{Corollary}
  
  \newtheorem{remark}[thm]{Remark}

\usepackage{latexsym}\usepackage[all]{xy}\usepackage{graphics}\usepackage{lscape}\usepackage{array}

\def\quot{/\!\!/}
\def\hom{\mathsf{Hom}}
\renewcommand{\leq}{\leqslant}

\renewcommand{\geq}{\geqslant}

\newcommand{\bG}{\mathbf{G}}
\newcommand{\X}{\mathfrak{X}}

\newcommand{\R}{\mathbb{R}}
\newcommand{\Z}{\mathbb{Z}}

\newcommand{\C}{\mathbb{C}}

\newcommand{\VV}{\mathbb{V}}

\newcommand{\N}{\mathbb{N}}

\newcommand{\SU}{\mathrm{SU}}
\newcommand{\U}{\mathrm{U}}
\newcommand{\GL}{\mathrm{GL}}

\newcommand{\SL}{\mathrm{SL}}
\newcommand{\SO}{\mathrm{SO}}
\newcommand{\Or}{\mathrm{O}}
\newcommand{\Sp}{\mathrm{Sp}}

\newcommand{\diag}{\mathrm{diag}}
\newcommand{\la}{\langle}
\newcommand{\ra}{\rangle}

\DeclareMathOperator{\Aut}{Aut}

\DeclareMathOperator{\tr}{tr}
\DeclareMathOperator{\Pf}{Pf}

\DeclareMathOperator{\End}{End}

\DeclareMathOperator{\Fix}{Fix}
\DeclareMathOperator{\Lie}{Lie}

\newcommand{\liep}{\mathfrak{p}}
\newcommand{\liepc}{\mathfrak{p}^{\mathbb{C}}}
\newcommand{\liek}{\mathfrak{k}}
\newcommand{\liekc}{\mathfrak{k}^{\mathbb{C}}}
\newcommand{\lieg}{\mathfrak{g}}
\newcommand{\liegc}{\mathfrak{g}^{\mathbb{C}}}

\title[Topology of Moduli of Free Group Representations]{Topology of Moduli Spaces of Free Group Representations in Real Reductive Groups}

\author[A. C. Casimiro]{A. C. Casimiro}

\address{Departamento Matem\'{a}tica, Faculdade de Ci\^{e}ncias e Tecnologia,
Universidade Nova de Lisboa}

\email{amc@fct.unl.pt}

\author[C. Florentino]{C. Florentino}

\address{Departamento Matem\'{a}tica, Instituto Superior T\'{e}cnico, Av. Rovisco
Pais, 1049-001 Lisbon, Portugal}
\email{cfloren@math.ist.utl.pt}

\author[S. Lawton]{S. Lawton}

\address{Department of Mathematics, The University of Texas-Pan American,
1201 West University Drive Edinburg, TX 78539, USA}

\email{lawtonsd@utpa.edu}

\author[A. Oliveira]{A. Oliveira}
  \address{Departamento de Matem\'atica \\
  Universidade de Tr\'{a}s-os-Montes e Alto Douro\\ UTAD \\
Quinta dos Prados \\ 5001-801 Vila Real \\ Portugal}
  \email{agoliv@utad.pt}

\thanks{This work was partially supported by the projects PTDC/MAT/099275/2008 and PTDC/MAT/120411/2010, FCT, Portugal.  The authors also acknowledge support from U.S. National Science Foundation grants DMS 1107452, 1107263, 1107367 ``RNMS: GEometric structures And Representation varieties" (the GEAR Network).  Additionally, the third author was partially supported by the Simons Foundation grant 245642 and the U.S. National Science Foundation grant DMS 1309376, and the fourth author was partially supported by Centro de Matem\'atica da Universidade de Tr\'as-os-Montes e Alto Douro (PEst-OE/MAT/UI4080/2011).}

\keywords{Character varieties, Real reductive groups, Representation varieties}

\makeatother

\begin{document}
\begin{abstract}
Let $G$ be a real reductive algebraic group with maximal compact
subgroup $K$, and let $F_{r}$ be a rank $r$ free group.
We show that the space of closed orbits in $\hom(F_{r},G)/G$
admits a strong deformation retraction to the orbit space $\hom(F_{r},K)/K$.
In particular, all such spaces have the same homotopy type. We compute
the Poincar\'{e} polynomials of these spaces for some low rank groups
$G$, such as $\Sp(4,\mathbb{R})$ and $\U(2,2)$. We also compare
these real moduli spaces to the real points of the corresponding complex
 moduli spaces, and describe the geometry of many examples.
\end{abstract}
\maketitle

\section{Introduction}

Let $G$ be a complex reductive algebraic group and $\Gamma$ be a
finitely generated group. Moduli spaces of representations of $\Gamma$
into $G$, the so-called $G$-character varieties of $\Gamma$, play
important roles in hyperbolic geometry, the theory of bundles and
connections, knot theory and quantum field theories. These are spaces
of the form $\mathfrak{X}_{\Gamma}(G):=\hom(\Gamma,G)\quot G$, where
the quotient is to be understood in the setting of (affine) geometric
invariant theory (GIT), for the conjugation action of $G$ on the
representation space $\hom(\Gamma,G)$.

Some particularly relevant cases include, for instance, the fundamental
group $\Gamma=\pi_{1}(X),$ of a compact Riemann surface $X$. In
this situation, character varieties can be identified, up to homeomorphism,
with certain moduli spaces of $G$-Higgs bundles over $X$ (\cite{H,S}).
Another important case is when $\Gamma=\pi_{1}(M\setminus L)$ where
$L$ is a knot (or link) in a 3-manifold $M$; here, character varieties
define important knot and link invariants, such as the A-polynomial
(\cite{CCGLS}).

In the case when $\Gamma$ is a free group $F_{r}$ of rank $r\geq1$,
the topology of $\mathfrak{X}_{r}(G):=\mathfrak{X}_{F_{r}}(G)$, in
this generality, was first investigated in \cite{Florentino-Lawton:2009}.
Note that we always have embeddings $\mathfrak{X}_{\Gamma}(G)\subset\mathfrak{X}_{r}(G)$,
since any finitely generated $\Gamma$ is a quotient of some free
group $F_{r}$.
With respect to natural Hausdorff topologies, the spaces $\mathfrak{X}_{r}(G)$
turn out to be homotopy equivalent to the quotient spaces $\mathfrak{X}_{r}(K):=\hom(F_{r},K)/K$,
where $K$ is a maximal compact subgroup of $G$. Moreover, there
is a canonical strong deformation retraction from $\mathfrak{X}_{r}(G)$
to $\mathfrak{X}_{r}(K)$. The proofs of these results use Kempf-Ness
theory, which relates, under certain conditions, the action of a compact
group $K$ on a complex algebraic variety, to the action of its complexification
$G=K^{\mathbb{C}}$.

In the present article, we extend these results to the more general
case when $G$ is a \emph{real reductive Lie group} (see Definition
\ref{def:condforG} below, for the precise conditions we consider).
Note that this situation includes both the compact case $G=K$, and
its complexification $G=K^{\mathbb{C}}$, since both are special cases
of real reductive groups, but also includes non-compact real groups
for which we cannot identify $G$ with the complexification of $K$.
As main examples, we have the split real forms of complex simple groups
such as $\SL(n,\mathbb{R})$, $\Sp(2n,\mathbb{R})$ and other classical
matrix groups.
For such groups $G$, the appropriate geometric structure on the analogous
GIT quotient, still denoted by $\mathfrak{X}_{r}(G):=\hom(F_{r},G)\quot G$
(and where $G$ again acts by conjugation) was considered by Richardson
and Slodowy in \cite{rich-slod:1990}. As in the complex case, this
quotient parametrizes closed orbits under $G$, but contrary to that
case, even when $G$ is algebraic, the quotient is in general only
a semi-algebraic set, in a certain real vector space.

One of our main results is that, with respect to the natural topologies
induced by natural embeddings in vector spaces, $\mathfrak{X}_{r}(G)$
is again homotopy equivalent to $\mathfrak{X}_{r}(K)$. As a Corollary,
we obtain a somewhat surprising result that the homotopy type of
$\mathfrak{X}_{r}(G)$ depends only on $r$ and on $K$, but not on
$G$. This is especially interesting when we have two distinct real
groups $G_{1}$ and $G_{2}$ sharing the same maximal compact $K$,
as it means that the $G_{1}$- and $G_{2}$-character varieties of
$F_{r}$ are equivalent, up to homotopy.

The second main result states that, when $G$ is algebraic, there
is also a strong deformation retraction from $\mathfrak{X}_{r}(G)$
to $\mathfrak{X}_{r}(K)$. The proofs of these statements use the Kempf-Ness theory
for real groups developed by Richardson and Slodowy in \cite{rich-slod:1990}.

It should be remarked, by way of contrast, that the homotopy equivalence
statement above does not hold for other finitely generated groups,
such as $\Gamma=\pi_{1}(X)$, for a Riemann surface $X$, even in
the cases $G=\SL(n,\mathbb{C})$ and $K=\SU(n)$ (\cite{BF}). On the
other hand, very recently it was shown by different techniques
that the deformation statements hold when $\Gamma$ is a finitely
generated Abelian group (\cite{Florentino-Lawton:2013b}), or a finitely generated nilpotent
group (\cite{B}).

Using these homotopy equivalences, we present new computations of Poincar\'{e} polynomials of some of
the character varieties considered, such as $\Sp(4,\mathbb{R})$ and $\U(2,2)$.

Lastly, when $G$ is a complex reductive algebraic group there are very explicit
descriptions of some of the spaces $\mathfrak{X}_{r}(G)$ in terms of natural
coordinates which we call {\em trace coordinates} (see Section \ref{sec:examples}).
Thus, for these examples, and taking the real points in these trace coordinates,
we obtain a concrete relation between
$\mathfrak{X}_{r}(G(\mathbb{R}))$ and $\mathfrak{X}_{r}(G)(\mathbb{R})$, which allows
the visualization of the deformation retraction.

The article can be outlined as follows. In Section 2, we
present the first definitions and properties of real reductive Lie groups
$G$, and of $G$-valued character varieties of free groups. In the third section, we use the polar/Cartan decomposition
to describe the deformation retraction of $\hom(F_{r},G)/K$ onto
$\mathfrak{X}_{r}(K)$.
The fourth section is devoted to describe the Kempf-Ness set for this context and to the proof of the main results: the
homotopy equivalence between $\mathfrak{X}_{r}(G)$ and $\mathfrak{X}_{r}(K)$,
and the canonical strong deformation retraction, in the case of algebraic
$G$.
In Section 5 we consider low rank orthogonal,
unitary and symplectic groups, and compute the Poincar\'{e}
polynomials of some $G$-character varieties, for non-compact $G$,
such as $\mathfrak{X}_{r}(\U(2,1))$, $\mathfrak{X}_{r}(\U(2,2))$,
$\mathfrak{X}_{r}(\Sp(4,\mathbb{R}))$ and $\mathfrak{X}_{r}(\SO(3,\mathbb{C}))$.
One crucial ingredient for these computations is the topology of $\mathfrak{X}_{r}(\U(2))$
which is based on T. Baird's determination of the Poincar\'{e} polynomial
of $\mathfrak{X}_{r}(\SU(2))$. Finally, Section 6 describes in detail the geometry of $\SL(2,\mathbb{R})$-character varieties in terms of natural invariant functions, such as trace coordinates, defined on the corresponding
$\SL(2,\mathbb{C})$-character variety.

\section{Real character varieties}

\label{sec:real-char}

\subsection{Setting}


Let us define the precise conditions on a real Lie group $G$, for which
our results will apply.
\begin{defn}
\label{def:condforG} Let $K$ be a compact
Lie group. We say that $G$ is a {\em real $K$-reductive Lie group} if the following conditions hold:
\begin{enumerate}
\item $K$ is a maximal compact subgroup of $G$;
\item $G$ is a subgroup, containing the identity component, of a linear
real algebraic group $\mathbf{G}(\R)$ defined as the $\R$-points
of a complex reductive algebraic group $\mathbf{G}$
defined over the field $\R$.
\item $G$ is Zariski dense in $\mathbf{G}$.
\end{enumerate}
\end{defn}

In other words, condition (ii) says that the Lie group $G$ is such that there exists a complex
reductive algebraic group $\mathbf{G}$, given by
the zeros of a set of polynomials with real coefficients, such that
\begin{equation}
\mathbf{G}(\R)_{0}\subseteq G\subseteq\mathbf{G}(\R),\label{eq1}\end{equation}
 where $\mathbf{G}(\R)$ denotes the real algebraic group of $\R$-points
of $\mathbf{G}$, and $\mathbf{G}(\R)_{0}$ its identity component.
We note that, if $G\neq\mathbf{G}(\R)$, then $G$ is not
necessarily an algebraic group (consider for example $G=\GL(n,\R)_{0}$).
When $K$ is understood, we often simply call $G$ a real reductive Lie group.

\begin{rem}\label{rem:1}
\begin{enumerate}
\item[]
\item Since $\mathbf{G}$ is a complex reductive algebraic group,
it is isomorphic to a closed subgroup of some $\GL(m,\C)$ (see
\cite[Theorem 2.3.7]{springer:2009}), so $\mathbf{G}(\R)$ is isomorphic
to a closed subgroup of some $\GL(n,\R)$ (ie, it is a linear algebraic group).
Hence, one can think of
both $\mathbf{G}$ and $G$ as Lie groups of matrices. Accordingly,
unless explicitly mentioned otherwise, we will consider on $\mathbf{G}$
and on $G$ the usual Euclidean topology which is induced from (and
is independent of) such an embedding.

\item Since $\mathbf{G}(\R)$ is a real algebraic group then, if it
is connected, $G=\mathbf{G}(\R)$ is algebraic and Zariski dense in
$\mathbf{G}$. So, condition (iii) in Definition \ref{def:condforG} holds automatically provided
that $\mathbf{G}(\R)$ is connected.

\item We point out that our working definition of real $K$-reductive group does not coincide with some definitions of a real reductive group encountered in the literature, such as, for instance,  \cite{knapp:2002}.
\end{enumerate}
\end{rem}

Denote by $\lieg$ the Lie algebra of $G$, and by $\Lie(\mathbf{G})$,
respectively $\Lie(\mathbf{G}(\R))$, the Lie algebras of $\mathbf{G}$,
and $\mathbf{G}(\R)$. It is clear that $\mathbf{G}(\R)$ is a
real form of $\mathbf{G}$, so that $\Lie(\mathbf{G})=\Lie(\mathbf{G}(\R))^{\C}$,
where $\Lie(\mathbf{G}(\R))^{\C}$ denotes the complexification of
$\Lie(\mathbf{G}(\R))$. It follows that $G$ is also a real form
of $\mathbf{G}$, since \eqref{eq1} implies that $\lieg=\Lie(\mathbf{G}(\R))$
and hence \[
\liegc=\Lie(\mathbf{G}).\]
 In view of this, we shall write $\liegc$ instead of $\Lie(\mathbf{G})$.

The given conditions on $G$ are not very restrictive. Indeed, all
classical real matrix groups are in this setting. On the other hand,
$G$ can also be any complex reductive Lie group, if we view it as
a real reductive Lie group in the usual way.

As an example which is
not under the conditions of Definition \ref{def:condforG}, we can
consider $\widetilde{\SL(n,\R)}$, the universal covering group of
$\SL(n,\R)$, which admits no faithful finite dimensional linear representation (and hence is not a matrix group).

\subsection{Character varieties}

Let $F_{r}$ be a rank $r$ free group and $\mathbf{G}$ be a complex reductive algebraic group
defined over $\mathbb{R}$.
The \emph{$\mathbf{G}$-representation
variety of $F_{r}$} is defined as \[
\mathfrak{R}_{r}(\mathbf{G}):=\hom(F_{r},\mathbf{G}).\]

There is a bijection between $\mathfrak{R}_{r}(\mathbf{G})$ and $\mathbf{G}^{r}$,
in fact this is a homeomorphism if $\mathfrak{R}_{r}(\mathbf{G})$
is endowed with the compact-open topology (as defined on a space of
maps, with $F_{r}$ given the discrete topology) and $\mathbf{G}^{r}$
with the product topology. As $\mathbf{G}$ is a smooth affine variety,
$\mathfrak{R}_{r}(\mathbf{G})$ is also a smooth affine variety and
it is defined over $\mathbb{R}$.

Consider now the action of $\mathbf{G}$ on $\mathfrak{R}_{r}(\mathbf{G})$
by conjugation. This
defines an action of $\mathbf{G}$ on the algebra $\mathbb{C}[\mathfrak{R}_{r}(\mathbf{G})]$
of regular functions on $\mathfrak{R}_{r}(\mathbf{G})$. Let $\mathbb{C}[\mathfrak{R}_{r}(\mathbf{G})]^{\mathbf{G}}$
denote the subalgebra of $\mathbf{G}$-invariant functions.
Since $\mathbf{G}$ is reductive the affine categorical
quotient may be defined as \[
\mathfrak{X}_{r}(\mathbf{G}):=\mathfrak{R}_{r}(\mathbf{G})/\!/\mathbf{G}=\mathrm{Spec}_{\max}(\mathbb{C}[\mathfrak{R}_{r}(\mathbf{G})]^{\mathbf{G}}).\]
 This is a singular affine variety (not necessarily irreducible), whose points correspond to unions
of $\mathbf{G}$-orbits in $\mathfrak{R}_{r}(\mathbf{G})$ whose Zariski
closures intersect. Since $\X_r(\bG)$ is an affine variety, it is a subset of an affine space, and inherits the Euclidean topology. With respect to this topology, in \cite{Florentino-Lawton:2013b}, it is shown that $\X_r(\bG)$ is homeomorphic to the conjugation orbit space of closed orbits (called the {\it polystable quotient}). $\mathfrak{X}_{r}(\mathbf{G})$, together with that topology, is called the \emph{$\mathbf{G}$-character variety}.

As above, let $K$ be a compact Lie group, and $G$ be a real $K$-reductive Lie group.
In like fashion, we define the \emph{$G$-representation variety of $F_{r}$}:
\[
\mathfrak{R}_{r}(G):=\hom(F_{r},G).
\]
Again, $\mathfrak{R}_{r}(G)$ is homeomorphic to $G^{r}$.
Similarly, as a set, we define
\[\mathfrak{X}_{r}(G):=\mathfrak{R}_{r}(G)/\!/G
\]
to be the set of closed orbits under the conjugation action of $G$
on $\mathfrak{R}_{r}(G)$. We give $\mathfrak{X}_{r}(G)$ the Hausdorff topology
induced by the quotient topology on $\mathfrak{R}_{r}(G)$. It is likewise
called the \emph{$G$-character variety of $F_r$} even though it may not even be a semi-algebraic set.
However, it is an affine real semi-algebraic set when G is real algebraic, and it is always Hausdorff
because we considered only closed $G$-orbits; see \cite{rich-slod:1990}.  This quotient coincides with the one considered by Richardson-Slodowy in \cite[Section 7]{rich-slod:1990}.

For $K$ a compact Lie group, with its usual topology, we also define the space \[
\mathfrak{X}_{r}(K):=\hom(F_{r},K)/K\cong K^{r}/K,\]
 called the \emph{$K$-character variety} of $F_{r}$, which is a compact and Hausdorff space
as the $K$-orbits are always closed. Moreover, it can be identified with a semi-algebraic subset of $\mathbb{R}^{d}$, for
some $d$.

Our aim in this paper is to compare the topologies of $\mathfrak{X}_{r}(G)$
and of $\mathfrak{X}_{r}(K)$, whenever $G$ is a real $K$-reductive Lie group, as in
Definition \ref{def:condforG}. In fact, we will show that $\mathfrak{X}_{r}(G)$
and $\mathfrak{X}_{r}(K)$ are generally homotopy equivalent, and when $G$
is further assumed to be algebraic, there is a natural strong deformation
retraction from $\mathfrak{X}_{r}(G)$ to $\mathfrak{X}_{r}(K)$.
The first step in that direction is the proof that there is a strong
deformation retraction of $\mathfrak{R}_{r}(G)/K$ onto $\mathfrak{X}_{r}(K)$.
This follows directly from the existence of a $K$-equivariant strong
deformation retraction of $G$ onto $K$, as will be explained in
the next section.

\section{Cartan decomposition and deformation to the maximal compact}

\label{sec:polar-dec}

We begin by recalling some facts on real Lie algebra theory.

\subsection{Cartan decomposition}\label{subsec:Cartan}

As before, let $\lieg$ denote the Lie algebra of $G$, and $\liegc$
the Lie algebra of $\mathbf{G}$. We will fix a Cartan involution
$\theta:\liegc\to\liegc$ which restricts to a Cartan involution \begin{equation}
\theta:\lieg\to\lieg,\label{eq:Cartan inv}\end{equation}
still denoted in the same way. Recall (see, for instance, \cite[Theorem 6.16]{knapp:2002}
and also \cite[Theorem 7.1]{helgason:2001}), that such $\theta$
is defined as $\theta:=\sigma\tau$, where $\sigma,\tau$ are involutions
of $\liegc$ that commute, and such that $\lieg=\Fix(\sigma)$ and
$\liek':=\Fix(\tau)$ is the compact real form of $\liegc$ (so that $\liek'$ is the Lie algebra of a maximal compact subgroup of $\mathbf{G}$).  See Remark \ref{concreteremark} for concrete descriptions of these involutions in the setting of our main theorems.

Our choice of $\theta$ yields a Cartan decomposition of $\lieg$
\begin{equation}
\lieg=\liek\oplus\liep\label{eq:Cartan decomp of g}\end{equation}
 where \[
\liek=\lieg\cap\liek',\hspace{0.5cm}\liep=\lieg\cap i\liek'\]
 and $\theta|_{\liek}=1$ and $\theta|_{\liep}=-1$. Furthermore,
$\liek$ is precisely the Lie algebra of a maximal compact subgroup
$K$ of $G$. Notice that $K=K'\cap G$, where $K'$ is a maximal compact
subgroup of $\mathbf{G}$, with Lie algebra $\liek'=\liek\oplus i\liep$.
Moreover, $\liek$ and $\liep$ are such that $[\liek,\liep]\subset\liep$ and $[\liep,\liep]\subset\liek$.
Of course, we also have a Cartan decomposition of $\liegc$: \begin{equation}
\liegc=\liekc\oplus\liepc\label{eq:Cartan decomp of gc}\end{equation}
 with $\theta|_{\liekc}=1$ and $\theta|_{\liepc}=-1$.

Recall also that the Cartan involution \eqref{eq:Cartan inv} lifts
to a Lie group involution \[
\Theta:G\to G\]
 whose differential is $\theta$ and such that $K=\Fix(\Theta)=\{g\in G:\, \Theta(g)=g\}$.

\subsection{A deformation retraction from $G$ onto $K$}

\label{polar}

The multiplication map \[
m:K\times\exp(\liep)\to G,\]
 provides a diffeomorphism $G\simeq K\times\exp(\liep)$ (see \cite[Theorem 2.2]{rich-slod:1990} and the references therein). In particular, the exponential is injective on $\liep$. The inverse $m^{-1}:G\to K\times\exp(\liep)$ is defined as \[
m^{-1}(g)=(g(\Theta(g)^{-1}g)^{-1/2},(\Theta(g)^{-1}g)^{1/2}).\]
Here, we notice that if $g\in \exp(\liep)$ then $\Theta(g)=g^{-1}$. If we write $g=k\exp(X)$, for some $k\in K$ and $X\in\liep$, then
\begin{align*}
\Theta(g)^{-1}g&=\Theta(k\exp(X))^{-1}k\exp(X)=(k\exp(-X))^{-1}k\exp(X)=\\&=\exp(-X)^{-1}k^{-1}k\exp(X)
=\exp(2X).\end{align*}
So define \[
(\Theta(g)^{-1}g)^{t}:=\exp\left(2tX\right),\]
for any real parameter $t$. From this, one concludes that the topology
of $G$ is determined by $K$. It is a known fact that there is a
$K$-equivariant strong deformation retraction from $G$ to $K$.
For completeness, we provide a proof.

Consider, for each $t\in[0,1]$, the continuous map $f_{t}:G\to G$
defined by \[
f_{t}(g)=g(\Theta(g)^{-1}g)^{-t/2}.\]
More precisely, if $g=k\exp(X)$, for some $k\in K$ and $X\in\liep$,
then $f_{t}(g)=k\exp((1-t)X).$
\begin{lem}
\label{lemma:exponents} Let $a\in\R$, $h\in K$ and $g\in G$. Then
$(h\Theta(g)^{-1}gh^{-1})^{a}=h(\Theta(g)^{-1}g)^{a}h^{-1}$.\end{lem}
\begin{proof}
When $a$ is an integer, this is obvious. Suppose that $a\in\R$.
Then, as noted above, $\Theta(g)^{-1}g=\exp(X)$ for some $X\in\liep$
and since the exponential is equivariant with respect to conjugation
and to the adjoint representation, we have \begin{align*}
(h\Theta(g)^{-1}gh^{-1})^{a}=\exp(hXh^{-1})^{a}=\exp(ahXh^{-1})=h\exp(aX)h^{-1}=\\
=h(\Theta(g)^{-1}g)^{a}h^{-1}.\end{align*}
\end{proof}
\begin{prop}
\label{sdr}The map $H:[0,1]\times G\to G$, $H(t,g)=f_{t}(g)$ is
a strong deformation retraction from $G$ to $K$, and for each $t$,
$H(t,-)=f_{t}$ is $K$-equivariant with respect to the action of
conjugation of $K$ in $G$. \end{prop}
\begin{proof}
Clearly $H|_{\{0\}\times G}=1_{G}$ and $H(\{1\}\times G)\subset K$.
Moreover, it is also clear that $H|_{\{t\}\times K}=1_{K}$ for all
$t$. This shows that $H$ is a strong deformation retraction from
$G$ to $K$. To prove that $f_{t}$ is $K$-equivariant, we see,
using Lemma \ref{lemma:exponents}, that for any $h\in K$, \[
\begin{split}f_{t}(hgh^{-1}) & =hgh^{-1}(\Theta(hgh^{-1})^{-1}hgh^{-1})^{-t/2}\\
 & =hgh^{-1}(h\Theta(g)^{-1}gh^{-1})^{-t/2}\\
 & =hg(\Theta(g)^{-1}g)^{-t/2}h^{-1}=hf_{t}(g)h^{-1}.\end{split}
\]

\end{proof}

By Proposition \ref{sdr}, there is a $K$-equivariant strong deformation
retraction from $G$ to $K$, so there is a $K$-equivariant strong
deformation retraction from $G^{r}$ onto $K^{r}$ with respect to
the diagonal action of $K$. This immediately implies:
\begin{cor}
\label{sdr1} Let $K$ be a compact Lie group and $G$ be a real $K$-reductive
Lie group. Then $\mathfrak{X}_{r}(K)$ is a strong deformation retract
of $\mathfrak{R}_{r}(G)/K$.
\end{cor}

\section{The Kempf-Ness set and the deformation retraction}

\label{sec:def-retract}As before, fix a compact Lie group $K$, and
a real $K$-reductive Lie group $G$. Suppose that $G$ acts linearly
on a complex vector space $\mathbb{V}$, equipped with a Hermitian
inner product $\la\,,\ra$. Without loss of generality we can
assume that $\la\,,\ra$ is $K$-invariant, by averaging.
\begin{defn}
A vector $X\in\VV$ is a \emph{minimal vector} for the action of $G$
in $\VV$ if \[
\|X\|\leq\|g\cdot X\|,\]
for every $g\in G$, where $||\cdot||$ is the norm corresponding
to $\la\,,\ra$. Let $\mathcal{KN}_{G}=\mathcal{KN}(G,\VV)$
denote the set of minimal vectors. $\mathcal{KN}_{G}$ is known as
the \emph{Kempf-Ness set} in $\VV$ with respect to the action of
$G$. Note that $\mathcal{KN}_{G}$ depends on the choice of $\la\,,\ra$.
\end{defn}
For each $X\in\VV$, define the smooth real valued function $F_{X}:G\to\R$
by \[
F_{X}(g)=\frac{1}{2}\|g\cdot X\|^{2}.\]
The following characterization of minimal vectors is given in \cite[Theorem 4.3]{rich-slod:1990}.
\begin{thm}
\label{thm:belong-to-KN}Let $X\in\VV$. The following conditions
are equivalent:

(1) $X\in\mathcal{KN}_{G}$;

(2) $F_{X}$ has a critical point at $1_{G}\in G$;

(3) $\la A\cdot X,X\ra=0$, for every $A\in\liep$.
\end{thm}
Since the action is linear and condition (3) above is polynomial,
we see that $\mathcal{KN}_{G}$ is a closed algebraic set in $\mathbb{V}$.
Kempf-Ness theory also works for closed $G$-subspaces. Indeed, let
$Y$ be an arbitrary closed $G$-invariant subspace of $\VV$, and 
define \[
\mathcal{KN}_{G}^{Y}:=\mathcal{KN}_{G}\cap Y.\]
Consider the map \[
\eta\,:\mathcal{KN}_{G}^{Y}/K\rightarrow Y\quot G,\]
obtained from the $K$-equivariant inclusion $\mathcal{KN}_{G}^{Y}\hookrightarrow Y$ and the natural map $Y/K\to Y\quot G$.

The next theorem is proved in \cite[Proposition 7.4, Theorems 7.6, 7.7 and 9.1]{rich-slod:1990}.
\begin{thm}
\label{thm:RS}The map $\eta:\mathcal{KN}_{G}^{Y}/K\to Y\quot G$
is a homeomorphism. In particular, if $Y$ is a real algebraic subset
of $\VV$, then $Y\quot G$ is homeomorphic to a closed semi-algebraic
set in some $\mathbb{R}^{d}$. Moreover, there is a $K$-equivariant
deformation retraction of $Y$ onto $\mathcal{KN}_{G}^{Y}$.
\end{thm}

\subsection{Kempf-Ness set for character varieties}

To apply the Kempf-Ness theorem to our situation, we need to embed
the $G$-invariant closed set $Y=\mathfrak{R}_{r}(G)=\hom(F_{r},G)\cong G^r$
in a complex vector space $\mathbb{V}$, as follows.
According to Remark \ref{rem:1}, we will assume, from now on, the following
commutative diagram of inclusions,
\begin{equation}
\begin{array}{ccccccc}
\mathrm{O}(n) & \subset & \GL(n,\mathbb{R}) & \subset & \GL(n,\mathbb{C}) & \subset & \mathfrak{gl}(n,\C)
\cong\mathbb{C}^{n^{2}}\\
\cup &  & \cup &  & \cup\\
K & \subset & G & \subset & \mathbf{G},\end{array}\label{eq:inclusions}
\end{equation}
where $G\subset\GL(n,\mathbb{R})$ is a closed subgroup. Note that the commuting square on the left
is guaranteed by one of the versions of the Peter-Weyl theorem (see, for example, \cite{knapp:2002}).

\begin{rem}\label{concreteremark}
As we consider $G$ embedded in some $\GL(n,\C)$ as a closed subgroup, the involutions $\tau,\sigma,\theta$ and $\Theta$, mentioned in Subsection \ref{subsec:Cartan}, become explicit. Indeed, under the inclusions $\mathfrak{g}\subset\mathfrak{gl}(n,\mathbb{R})$,
$\mathfrak{g}^{\mathbb{C}}\subset\mathfrak{gl}(n,\mathbb{C})$ and $G\subset \mathrm{GL}(n,\mathbb{R})$ we have $\tau(A)=-A^{*}$, where
$*$ denotes transpose conjugate, and $\sigma(A)=\bar{A}$. Hence, the Cartan involution is given by $\theta(A)=-A^{t}$, so that $\Theta(g)=(g^{-1})^{t}$. From now on, we will use these particular involutions.
\end{rem}

From \eqref{eq:inclusions} we obtain the embedding of $K^{r}$ ($r\in\mathbb{N}$)
into the vector space given by the product of the spaces of all $n$-square
complex matrices, which we denote by $\VV$: \[
\mathfrak{gl}(n,\C)^{r}\cong\C^{rn^{2}}=:\VV.\]
The adjoint representation of $\GL(n,\C)$ in $\mathfrak{gl}(n,\C)$
restricts to a representation \[
G\to\Aut(\VV)\]
given by \begin{equation}
g\cdot(X_{1},\ldots,X_{r})=(gX_{1}g^{-1},\ldots,gX_{r}g^{-1}),\hspace{0.2cm}g\in G,\ X_{i}\in\mathfrak{gl}(n,\C).\label{eq:G action}\end{equation}
Moreover, \eqref{eq:G action} yields a representation \[
\lieg\to\End(\VV)\]
of the Lie algebra $\lieg$ of $G$ in $\VV$ given by the Lie brackets:
\begin{equation}
A\cdot(X_{1},\ldots,X_{r})=(AX_{1}-X_{1}A,\ldots,AX_{r}-X_{r}A)=([A,X_{1}],\ldots,[A,X_{r}])\label{eq:rep of Lie(G)}\end{equation}
for every $A\in\lieg$ and $X_{i}\in\mathfrak{gl}(n,\C)$. In what
follows, the context will be clear enough to distinguish the notations
\eqref{eq:G action} and \eqref{eq:rep of Lie(G)}.

We choose a inner product $\la\,,\ra$ in $\mathfrak{gl}(n,\C)$
which is $K$-invariant, under the restriction of the representation
$\GL(n,\C)\to\Aut(\mathfrak{gl}(n,\C))$ to $K$. From this we obtain
a inner product on $\VV$, $K$-invariant by the corresponding diagonal
action of $K$: \begin{equation}
\la(X_{1},\ldots,X_{r}),(Y_{1},\ldots,Y_{r})\ra=\sum_{i=1}^{r}\la X_{i},Y_{i}\ra\label{eq: inner product on VV K-invariant}\end{equation}
for $X_{i},Y_{j}\in\mathfrak{gl}(n,\C)$. In $\mathfrak{gl}(n,\C)$,
$\la\,,\,\ra$ can be given explicitly by $\la A,B\ra=\tr(A^{*}B)$.

We can now prove one of the main results.
\begin{thm}
\label{main} The spaces $\mathfrak{X}_{r}(G)=\mathfrak{R}_{r}(G)/\!/G$ and $\mathfrak{X}_{r}(K)=\hom(F_{r},K)/K\cong K^{r}/K$
have the same homotopy type.\end{thm}
\begin{proof}
By the strong deformation retraction from Corollary \ref{sdr1}, we
have that $\mathfrak{X}_{r}(K)$ and $\mathfrak{R}_{r}(G)/K$ have
the same homotopy type. From Theorem \ref{thm:RS}, putting $Y=G^{r}=\mathfrak{R}_{r}(G)\subset \VV$,
we deduce that $\mathfrak{R}_{r}(G)/K$ and $\mathcal{KN}_{G}^{Y}/K$
have the same homotopy type. Again by Theorem \ref{thm:RS}, we also
have that $\mathcal{KN}_{G}^{Y}/K$ is homeomorphic to $\mathfrak{X}_{r}(G)$.
Now, as homotopy equivalence is transitive, we conclude that $\mathfrak{X}_{r}(K)$
and $\mathfrak{X}_{r}(G)$ have the same homotopy type.
\end{proof}

\begin{cor}
The homotopy type of the space $\mathfrak{X}_{r}(G)$ depends only
on the maximal compact subgroup $K$ of $G$. In other words, given
two real Lie groups $G_{1}$ and $G_{2}$ verifying our assumptions,
and which have isomorphic maximal compact subgroups, then $\mathfrak{X}_{r}(G_{1})$
and $\mathfrak{X}_{r}(G_{2})$ have the same homotopy type. \end{cor}
\begin{proof}
Since both $G_{1}$ and $G_{2}$ have $K$ as maximal compact, this
follows immediately from Theorem \ref{main}.
\end{proof}

\subsection{Deformation retraction from $\mathfrak{X}_{r}(G)$ onto $\mathfrak{X}_{r}(K)$}

Now, we want to show that $\mathfrak{X}_{r}(K)$ is indeed a deformation
retraction of $\mathfrak{X}_{r}(G)$.

For the $G$-invariant space $Y=\mathfrak{R}_{r}(G)\cong G^{r}$,
the Kempf-Ness set $\mathcal{KN}_{G}^{Y}\subset\mathbb{V}$, includes
the $K$-invariant subspace $Y=\hom(F_{r},K)\cong K^{r}$, and can
be characterized in concrete terms as follows.
\begin{prop}
\label{pro:general-KN}For $Y=\mathfrak{R}_{r}(G)\cong G^{r}\subset\mathbb{V}$,
the Kempf-Ness set is the closed set given by: \[
\mathcal{KN}_{G}^{Y}=\left\{(g_{1},\cdots,g_{r})\in G^{r}:\ \sum_{i=1}^{r}g_{i}^{*}g_{i}=\sum_{i=1}^{r}g_{i}g_{i}^{*}\right\}.\]
In particular, since $K$ is precisely the fixed set of the Cartan
involution, we have the inclusion $K^{r}\cong\hom(F_{r},K)\subset\mathcal{KN}_{G}^{Y}$.
The Kempf-Ness set is a real algebraic set, when $G$ is algebraic.\end{prop}
\begin{proof}
By Theorem \ref{thm:belong-to-KN} (3), an element $g=(g_{1},\cdots,g_{r})\in G^{r}\subset\GL(n,\mathbb{C})^{r}$
is in the Kempf-Ness set, if and only if \[
\la A\cdot g,g\ra=\left\langle \left ([A,g_{1}],\cdots,[A,g_{r}]\right ),\,(g_{1},\cdots,g_{r})\right\rangle =0,\]
for every $A\in\liep$ (see \eqref{eq:Cartan decomp of g}), where we used formula \eqref{eq:rep of Lie(G)}.
Using \eqref{eq: inner product on VV K-invariant}, this means that,
for all $A\in\liep$, we have
\begin{eqnarray*}
0 & = & \left\langle \left ( [A,g_{1}],\cdots,[A,g_{r}]\right ),\,(g_{1},\cdots,g_{r})\right\rangle =\sum_{i=1}^{r}\left\langle Ag_{i}-g_{i}A,\, g_{i}\right\rangle =\\
 & = & \sum_{i=1}^{r}\left(\left\langle Ag_{i},\, g_{i}\right\rangle -\left\langle g_{i}A,\, g_{i}\right\rangle\right) =\sum_{i=1}^{r}\left(\tr(g_{i}^{*}A^{*}g_{i})-\tr(A^{*}g_{i}^{*}g_{i})\right)=\\
 & = & \sum_{i=1}^{r}\left(\tr(A^{*}g_{i}g_{i}^{*})-\tr(A^{*}g_{i}^{*}g_{i})\right)=\sum_{i=1}^{r}\left\langle A,g_{i}g_{i}^{*}-g_{i}^{*}g_{i}\right\rangle \\
 & = & \left\langle A,\sum_{i=1}^{r}(g_{i}g_{i}^{*}-g_{i}^{*}g_{i})\right\rangle \end{eqnarray*}
 Here, we used bilinearity of $\la\,,\ra$ and the
cyclic permutation property of the trace. In fact, the last expression should
vanish for all $A\in\lieg=\liek\oplus\liep$ (by $K$-invariance of
the norm, the vanishing for $A\in\liek$ is automatic). So, since
$\la\,,\ra$ is a nondegenerate pairing, we conclude
that $X\in\mathcal{KN}_{G}^{Y}$ if and only if \[
\sum_{i=1}^{r}(g_{i}g_{i}^{*}-g_{i}^{*}g_{i})=0, \, \forall (g_1,\cdots g_r)\in G^r \]
as wanted. The last two sentences are immediate consequences.
\end{proof}

Recall that a matrix $A\subset\GL(n,\mathbb{C})$ is called \emph{normal}
if $A^{*}A=AA^{*}$. So, when $r=1$, the previous proposition says
that $\mathcal{KN}_{G}^{Y}=\{g\in G:\ g\mbox{ is normal}\}\subset \mathfrak{R}_{1}(G)$.

The following characterization, then follows directly from Theorem \ref{thm:RS}.
\begin{prop}
\label{pro:r1-KN}
When $r=1$, the character variety $\mathfrak{X}_{1}(G)=G\quot G$
is homeomorphic to the orbit space of the set of normal matrices in
$G$, under conjugation by $K$.
\end{prop}

Now, to prove that $\mathfrak{X}_{r}(K)$ is  a deformation
retraction of $\mathfrak{X}_{r}(G)$ we need to further assume, due to a technical point, that $G$ is algebraic. First, we have the following lemma.
\begin{lem}\label{subcomplex}
Assume that $G$ and $K$ are as before, and furthermore that $G$ is a real algebraic set.
There is a natural inclusion of finite CW-complexes $\mathfrak{X}_{r}(K)\subset\mathfrak{X}_{r}(G)$. \end{lem}
\begin{proof}
We need to show that the natural composition: \[
\mathfrak{X}_{r}(K)\to\mathfrak{R}_{r}(G)/K\to\mathfrak{X}_{r}(G)\]
does not send two distinct $K$-orbits to a single $G$-orbit.
This follows by Remark 4.7 in \cite{Florentino-Lawton:2013a}, and the polar decomposition discussed in Subsection \ref{polar}.
However, we prove it directly as follows.  It is equivalent to showing that given $\rho_{1},\rho_{2}\in\mathfrak{R}_{r}(K)$
such that $\rho_{2}=g\cdot\rho_{1}$ for some $g\in G$, then $\rho_{1}$
and $\rho_{2}$ are in the same $K$-orbit. Using $Y=\mathfrak{R}_{r}(G)$,
we have $\mathfrak{R}_{r}(K)\subset\mathcal{KN}_{G}^{Y}$ by Proposition
\ref{pro:general-KN}. On the other hand, Richardson-Slodowy showed
that $G\cdot\rho_{1}\cap\mathcal{KN}_{G}^{Y}=K\cdot\rho_{1}$ (Theorem
4.3 in \cite{rich-slod:1990}), which is enough to prove that $\mathfrak{X}_{r}(K) \subset \mathfrak{X}_{r}(G)$. We know that $\mathfrak{X}_{r}(K)$ is a semi-algebraic set and is closed. By Theorem \ref{thm:RS},  since we have assumed $G$ is algebraic, $\mathfrak{X}_{r}(G)$ is also a semi-algebraic set and closed. Thus, $\mathfrak{X}_{r}(K)$ can be considered as a semi-algebraic subset of $\mathfrak{X}_{r}(G)$. It is known that all semi-algebraic sets are cellular. Now, from page 214 of \cite{BCR98} we get that $\mathfrak{X}_{r}(K)$ is a sub-complex of $\mathfrak{X}_{r}(G)$.
\end{proof}

\begin{thm}\label{thm:maintheorem}
There is a strong deformation retraction from $\mathfrak{X}_{r}(G)$
to $\mathfrak{X}_{r}(K)$. \end{thm}
\begin{proof}
Proposition \ref{pro:general-KN} implies the following diagram is commutative:
\[
\begin{array}{cccc}
K^r/K & \stackrel{i}{\hookrightarrow} & G^r/K\\
\phi \hookdownarrow &  & \Vert\\
\mathcal{KN}_{G}^{G^r}/K &  \stackrel{j}{\hookrightarrow} & G^r/K\end{array}
\]

Corollary \ref{sdr1} and Theorem \ref{thm:RS}, imply the maps $i$ and $j$
induce isomorphisms on all homotopy groups; that is,
\[
i_{n}\,:\,\pi_{n}(K^r/K)\,\longrightarrow\, \pi_n(G^r/K) ~\, \text{~and~}\,
~ j_{n}\,:\, \pi_n(\mathcal{KN}_{G}^{G^r}/K)\,\longrightarrow\, \pi_n(G^r/K)
\]
are isomorphisms for all $n\,\geq\, 0$.

Thus, $\phi$ induces isomorphisms on all homotopy groups as well since $i=j\circ \phi$. Then, Lemma \ref{subcomplex} and Whitehead's theorem (see \cite[Theorem 4.5]{Hat}) imply $K^r/K$ is a strong deformation retraction of $\mathcal{KN}_{G}^{G^r}/K\cong G^r\quot G$.
\end{proof}

\begin{remark}
In \cite{knapp:2002}, Theorem 6.31, it is shown that any semisimple Lie group $G$, even those not considered in this paper like $\widetilde{\SL(2,\R)}$, admits a Cartan involution.  The fixed subspace of this involution defines a subgroup $H$, and $G$ deformation retracts onto $H$ equivariantly with respect to conjugation by $H$.  Now $H$ is not always the maximal compact; in fact it is, if and only if the center of $G$ is finite.  Nevertheless, we conclude that  $G^r/H$ is homotopic to $H^r/H$.

Additionally, it is a general result, proven independently by Malcev and later by Iwasawa $($\cite[Thm. 6]{Iwas}$)$, that every connected Lie group $G$, deformation retracts onto a maximal compact subgroup $H$ $($we thank the referee for the reference$)$.  If this deformation is $H$-equivariant, we likewise conclude $G^r/H$ is homotopic to $H^r/H$.

Either way, it would be interesting try to compare $G^r/G$ and $G^r/H$ without Kempf-Ness Theory available in these more general situations.
\end{remark}

\section{Poincar\'{e} Polynomials}

In this section, we describe the topology of some character varieties
and compute their Poincar\'{e} polynomials.

\subsection{Low rank unitary groups}

\begin{prop}
\label{prop:U(n)SU(n)} For any $r,n\in\N$, the following isomorphisms
hold:
\begin{itemize}
\item $\mathfrak{X}_{r}(\U(n))\cong\mathfrak{X}_{r}(\SU(n))\times_{(\Z/n\Z)^{r}}\mathrm{U}(1)^{r}$;
\item $\mathfrak{X}_{r}(\Or(n))\cong\mathfrak{X}_{r}(\SO(n))\times(\Z/2\Z)^{r}$,
if $n$ is odd.
\end{itemize}
\end{prop}

\proof The first item is proved in \cite[Theorem 2.4]{Florentino-Lawton:2012}.
The second item follows because, for $n$ odd, $\Or(n)$ is isomorphic
to $\SO(n)\times\Z/2\Z$. \endproof

In what follows, we consider cohomology with rational coefficients.

The Poincar\'{e} polynomial of $\mathfrak{X}_{r}(\SU(2))$ was calculated
by T. Baird in \cite[Theorem 7.2.4]{baird:2008}, using methods of
equivariant cohomology. His result is that \begin{equation}
P_{t}(\mathfrak{X}_{r}(\SU(2)))=1+t-\frac{t(1+t^{3})^{r}}{1-t^{4}}+\frac{t^{3}}{2}\left(\frac{(1+t)^{r}}{1-t^{2}}-\frac{(1-t)^{r}}{1+t^{2}}\right).\label{eq:Poincare poly Baird}\end{equation}

 From Proposition \ref{prop:U(n)SU(n)} and general results concerning finite quotients and rational cohomology
(see, for example  \cite{Bredon-transf}) we conclude that
\begin{equation}
\label{finite-quot}
H^{*}(\mathfrak{X}_{r}(\U(n))) \cong H^{*}(\mathfrak{X}_{r}(\SU(n))\times\mathrm{U}(1)^{r})^{(\Z/n\Z)^{r}}.
\end{equation}

For $n=2$, we now show the action on cohomology is trivial.  The argument for this theorem was suggested to us by T. Baird.

\begin{thm}\label{trivialaction}
The action of $(\Z/2\Z)^{r}$ on $H^{*}(\mathfrak{X}_{r}(\SU(2)))$ is trivial.
\end{thm}

\begin{proof}
Let $\Gamma$ be a finite subgroup of a connected Lie group $G$, and let $G$ act on a space $X$. If we restrict the action to $\Gamma$ acting on $X$, then the induced action of $\Gamma$ on $H^*(X)$ is trivial.  This is because for any element $\gamma \in \Gamma$, the corresponding automorphism of $X$ is homotopic to the identity map (take any path from the identity in $G$ to $\gamma$ to obtain the homotopy).

Therefore, since the action of $(\Z/2\Z)^{r}$ is the restriction of the action of the path-connected group $\SU(2)^r$ acting by multiplication, we conclude that the quotient map induces an isomorphism $$H^*(\SU(2)^r) \cong H^*(\SU(2)^r/(\Z/2\Z)^r).$$

Next, let $X, Y$ be two $G$-spaces and let $h: X \to Y$ be a $G$-equivariant map which induces an isomorphism in cohomology $H^*(X) \cong H^*(Y)$. Then $h$ also induces an isomorphism in equivariant cohomology $H^*_G(X)\cong H^*_G(Y)$; recall that $H^*_G(X):=H^*( EG \times_G X)$.  See \cite{ecnotes} for generalities on equivariant cohomology, and in particular a proof of this fact (Theorem 83, page 52).

Since $\Z/2\Z$ is central (the center of $\SU(2)$ is isomorphic to $\Z/2\Z$), the quotient map with respect to the $(\Z/2\Z)^r$-action on $\SU(2)^r$ is $\SU(2)$-equivariant (with respect to conjugation). It follows then that $$H^*_{\SU(2)}(\SU(2)^r) \cong H^*_{\SU(2)}(\SU(2)^r/(\Z/2\Z)^r)\cong H^*_{\SU(2)}(\SU(2)^r)^{(\Z/2\Z)^r},$$ where the second isomorphism holds for the same reason as \eqref{finite-quot}.  Therefore, the action is trivial on $H^*_{\SU(2)}(\SU(2)^r)$.

Let $G=\SU(2)$, and let $Y$ be the set of those tuples in $G^r$ that lie in a common maximal torus. The same homotopy argument in the previous paragraph shows the action is trivial on $H^*_G(Y)$ since $Y\subset G^r$. Moreover, considering the pair $(G^r,Y)$ and \cite[Prop. 2.19]{Hat}, we likewise conclude that the $(\Z/2\Z)^r$-action is trivial on $H^*_G(G^r,Y)$. We now recall a natural diagram in \cite{baird:2008}; diagram (7.8):
$$\xymatrix{\cdots \ar[r] & H^*_G(G^r,Y)\ar[r]& H^*_G(G^r)\ar[r]  & H^*_G(Y)\ar[r] &\cdots \\
            \cdots \ar[r]&H^*(G^r/G,Y/G)\ar[r]\ar[u]&H^*(G^r/G)\ar[u]\ar[r]&H^*(Y/G)\ar[u]\ar[r]&\cdots
}
$$
We just showed that the $(\Z/2\Z)^r$-action on all spaces in the top row is trivial.  In \cite{baird:2008}, equation (7.4) on page 59 shows that $H^*_G(G^r, Y)\cong H^*(G^r/G,Y/G)$; establishing the action is trivial on $H^*(G^r/G,Y/G)$ since the map $ H^*(G^r/G,Y/G)\to H^*_G(G^r, Y)$ is $(\Z/2\Z)^r$-equivariant.  The maps on the bottom row, $H^*(G^r/G,Y/G)\to H^*(G^r/G)$, are equivariant and surjective in positive degrees (see the proof of Theorem 7.2.4 in \cite{baird:2008}); which then implies the $(\Z/2\Z)^r$-action is trivial on $H^*(G^r/G)$ in positive degrees.  However, in degree 0 the $(\Z/2\Z)^r$-action is also trivial since $G^r/G$ is connected.  Thus, the action is trivial on $H^*(G^r/G)$, as required.
\end{proof}


Thus, $H^{*}(\mathfrak{X}_{r}(\U(2))) \cong H^{*}(\mathfrak{X}_{r}(\SU(2)))\otimes H^*(\mathrm{U}(1)^{r})^{(\Z/2\Z)^{r}}.$  However, the action of $(\Z/2\Z)^{r}$ on $H^{*}(\mathrm{U}(1)^{r})$ is the action of -1 on the circle, which is rotation by 180 degrees.  Rotation by 180 is homotopic to the identity, and thus the action is trivial on cohomology (the first paragraph in the above proof shows that this part generalizes to $\SU(n)$).  In other words, $H^{*}(\mathrm{U}(1)^{r})^{(\Z/2\Z)^{r}}=H^{*}(\mathrm{U}(1)^{r})$, and we conclude that \begin{equation}
H^{*}(\mathfrak{X}_{r}(\U(2)))\cong H^{*}(\mathfrak{X}_{r}(\SU(2)))\otimes H^{*}(\mathrm{U}(1)^{r}).\label{eq:equal cohomo}\end{equation}

\begin{prop}
\label{prop:Poincare pol of U(2)} The Poincar\'{e} polynomial of $\mathfrak{X}_{r}(\U(2))$
is the following: \[
P_{t}(\mathfrak{X}_{r}(\U(2)))=(1+t)^{r+1}-\frac{t(1+t+t^{3}+t^{4})^{r}}{1-t^{4}}+\frac{t^{3}}{2}\left(\frac{(1+t)^{2r}}{1-t^{2}}-\frac{(1-t^{2})^{r}}{1+t^{2}}\right).\]
 \end{prop}
\begin{proof}
This follows from the isomorphism \eqref{eq:equal cohomo}, from \eqref{eq:Poincare poly Baird} and the fact that the Poincar\'e polynomial of a circle is $1+t$.
\end{proof}

Now, take $G=\U(p,q)$, the group of automorphisms of $\C^{p+q}$ preserving
a nondegenerate hermitian form with signature $(p,q)$. In matrix
terms, one can write \[
\U(p,q)=\{M\in\GL(p+q,\C)\,|\, M^{*}I_{p,q}M=I_{p,q}\}\]
where
\[I_{p,q}=\left(\begin{array}{ccc}
I_{p} & 0\\
0 & -I_{q}\end{array}\right).\]
Its maximal compact is $K=\U(p)\times\U(q)$ and it embeds diagonally
in $\U(p,q)$: \[
(M,N)\hookrightarrow\left(\begin{array}{cc}
M & 0\\
0 & N\end{array}\right).\]
It follows from this that, as a subspace of $\mathfrak{X}_{r}(\U(p,q))$,
$\mathfrak{X}_{r}(\U(p)\times\U(q))$ is homeomorphic to $\mathfrak{X}_{r}(\U(p))\times\mathfrak{X}_{r}(\U(q))$.

From Theorem \ref{thm:maintheorem} and from Proposition
\ref{prop:Poincare pol of U(2)}, we have the following:
\begin{prop}
For any $p,q\geq1$ and any $r\geq1$, there exists a strong deformation
retraction from $\mathfrak{X}_{r}(\U(p,q))$ onto $\mathfrak{X}_{r}(\U(p))\times\mathfrak{X}_{r}(\U(q))$.
In particular, the Poincar\'{e} polynomials of $\mathfrak{X}_{r}(\U(2,1))$
and $\mathfrak{X}_{r}(\U(2,2))$ are given respectively by: \[
P_{t}(\mathfrak{X}_{r}(\U(2,1)))=P_{t}(\mathfrak{X}_{r}(\U(2)))(1+t)^r\]
 and \[
P_{t}(\mathfrak{X}_{r}(\U(2,2)))=P_{t}(\mathfrak{X}_{r}(\U(2)))^2.\]
\end{prop}

Exactly in the same way, since $\U(2)$ is a maximal compact subgroup
of $\Sp(4,\R)$ and of $\GL(2,\C)$, we have the following:
\begin{prop}
For any $r\geq1$, there exists a strong deformation retraction from
$\mathfrak{X}_{r}(\Sp(4,\R))$ and from $\mathfrak{X}_{r}(\GL(2,\C))$ onto $\mathfrak{X}_{r}(\U(2))$. In
particular, the Poincar\'{e} polynomials of $\mathfrak{X}_{r}(\Sp(4,\R))$ and $\mathfrak{X}_{r}(\GL(2,\C))$ are such that: \[
P_{t}(\mathfrak{X}_{r}(\Sp(4,\R)))=P_{t}(\mathfrak{X}_{r}(\GL(2,\C)))=P_{t}(\mathfrak{X}_{r}(\U(2))).\]

\end{prop}

\subsection{Low rank orthogonal groups}
\begin{prop}
$\mathfrak{X}_{r}(\SU(2))/(\Z/2\Z)^{r}\cong \mathfrak{X}_{r}(\SO(3))$
\end{prop}

\begin{proof}
$\SU(2)\to\SO(3)$ is the universal cover of $\SO(3)$ with fiber $\Z/2\Z\cong\pi_{1}(\SO(3))$.  The deck group is given by multiplication by minus the identity matrix.  This induces a $(\Z/2\Z)^{r}$-cover $\SU(2)^{r}\to\SO(3)^{r}$.  The corresponding $(\Z/2\Z)^{r}$-action is equivariant with respect to the conjugation action of $\SO(3)$ since $\Z/2\Z$ is acting by multiplication by central elements in each factor.  Therefore, $$(\SU(2)^{r}/\SO(3))/(\Z/2\Z)^{r}\cong \mathfrak{X}_{r}(\SO(3)),$$ where $\SO(3)\cong\mathrm{PSU}(2)$
acts diagonally by conjugation on $\SU(2)^{r}$. However, since $\mathrm{PSU}(2)\cong\SU(2)/Z(\SU(2))$, it is clear that $\SU(2)^{r}/\SO(3)\cong\mathfrak{X}_{r}(\SU(2))$.
\end{proof}

From this result we conclude that the cohomology of $\mathfrak{X}_{r}(\SO(3))$
is the $(\Z/2\Z)^{r}$-invariant part of the cohomology of $\mathfrak{X}_{r}(\SU(2))$:
\begin{equation}
H^{*}(\mathfrak{X}_{r}(\SO(3)))\cong H^{*}(\mathfrak{X}_{r}(\SU(2)))^{(\Z/2\Z)^{r}}.\label{eq:invariant cohomology of SU(2)}\end{equation}

\begin{prop}
\label{prop:Poincare pol of SO(3)} The Poincar\'{e} polynomials of $\mathfrak{X}_{r}(\SO(3))$
and of $\mathfrak{X}_{r}(\Or(3))$ are the following: \[
P_{t}(\mathfrak{X}_{r}(\SO(3)))=P_{t}(\mathfrak{X}_{r}(\SU(2)))\]
 and \[
P_{t}(\mathfrak{X}_{r}(\Or(3)))=2^rP_{t}(\mathfrak{X}_{r}(\SU(2))).\]

\end{prop}
\proof

The formula for $P_{t}(\mathfrak{X}_{r}(\SO(3)))$ follows from Theorem \ref{trivialaction}, and isomorphism
\eqref{eq:invariant cohomology of SU(2)}.
The formula for $P_{t}(\mathfrak{X}_{r}(\Or(3)))$ is immediate from
the one of $P_{t}(\mathfrak{X}_{r}(\SO(3)))$ and from Proposition \ref{prop:U(n)SU(n)}.
\endproof

Take $G=\SO(p,q)$, the group of volume preserving
automorphisms of $\R^{p+q}$ preserving a nondegenerate symmetric
bilinear form with signature $(p,q)$. In matrix terms, one can write
\[
\SO(p,q)=\{M\in\SL(p+q,\R)\,|\, M^{t}I_{p,q}M=I_{p,q}\}\]
 where \[
I_{p,q}=\left(\begin{array}{ccc}
-I_{p} & 0\\
0 & I_{q}\end{array}\right).\]
 If $p+q\geq3$, $\SO(p,q)$ has two connected components. Denote
by $\SO_{0}(p,q)$ the component of the identity.

The maximal compact subgroup of $\SO_{0}(p,q)$ is $K=\SO(p)\times\SO(q)$
and it embeds diagonally in $\SO(p,q)$. So, as in the case of $\U(p,q)$
mentioned above, it follows that, as a subspace of $\mathfrak{X}_{r}(\SO_{0}(p,q))$,
$\mathfrak{X}_{r}(\SO(p)\times\SO(q))$ is homeomorphic to $\mathfrak{X}_{r}(\SO(p))\times\mathfrak{X}_{r}(\SO(q))$.

From Theorem \ref{thm:maintheorem} and Proposition \ref{prop:Poincare pol of SO(3)}, we have thus the following:
\begin{prop}
For any $p,q\geq1$ and any $r\geq1$, there exists a strong deformation
retraction from $\mathfrak{X}_{r}(\SO_{0}(p,q))$ onto $\mathfrak{X}_{r}(\SO(p))\times\mathfrak{X}_{r}(\SO(q))$.
 In particular, the Poincar\'{e} polynomials of $\mathfrak{X}_{r}(\SO_{0}(2,3))$
and of $\mathfrak{X}_{r}(\SO_{0}(3,3))$ are given respectively by
\[
P_{t}(\mathfrak{X}_{r}(\SO_{0}(2,3)))=P_{t}(\mathfrak{X}_{r}(\SU(2)))(1+t)^r\]
 and \[
P_{t}(\mathfrak{X}_{r}(\SO_{0}(3,3)))=P_{t}(\mathfrak{X}_{r}(\SU(2)))^2.\]
\end{prop}

In the same way, since $\SO(3)$ (resp. $\Or(3)$) is a maximal compact subgroup
of both $\SL(3,\R)$ (resp. $\GL(3,\R)$) and $\SO(3,\C)$ (resp. $\Or(3,\C)$), we have the following:
\begin{prop}
For any $r\geq1$, there exists a strong deformation retraction from
$\mathfrak{X}_{r}(\SL(3,\R))$ and $\mathfrak{X}_{r}(\SO(3,\C))$ onto $\mathfrak{X}_{r}(\SO(3))$ and
from $\mathfrak{X}_{r}(\GL(3,\R))$ and $\mathfrak{X}_{r}(\Or(3,\C))$ onto $\mathfrak{X}_{r}(\Or(3))$.
In particular, the Poincar\'{e} polynomials of $\mathfrak{X}_{r}(\SL(3,\R))$ and $\mathfrak{X}_{r}(\SO(3,\C))$ are equal and given by:
\[P_t(\mathfrak{X}_{r}(\SL(3,\R)))=P_{t}(\mathfrak{X}_{r}(\SO(3,\C)))=P_{t}(\mathfrak{X}_{r}(\SU(2))).\]
Similarly, the Poincar\'{e} polynomials of $\mathfrak{X}_{r}(\GL(3,\R))$ and $\mathfrak{X}_{r}(\Or(3,\C))$ are equal and given by:
\[P_t(\mathfrak{X}_{r}(\GL(3,\R)))=P_{t}(\mathfrak{X}_{r}(\Or(3,\C)))=2^rP_{t}(\mathfrak{X}_{r}(\SU(2))).\]
\end{prop}

\section{Comparing Real and Complex character varieties}

In this section, we slightly change the perspective. Instead of comparing
the topologies of $K$- and $G$- character varieties, we present some
results on the relation between the topology and \emph{geometry} of
the character varieties $\mathfrak{X}_{r}(G)$ and (the real points
of) $\mathfrak{X}_{r}(\mathbf{G})$, making explicit use of \emph{trace
coordinates}. These coordinates have been previously considered in
the literature, and serve to embed $\mathfrak{X}_{r}(\mathbf{G})$
in complex vector spaces. Then, we provide a detailed analysis of
some examples (real forms $G$ of $\mathbf{G}=\SL(2,\mathbb{C})$),
showing how the geometry of these spaces compare, and how to understand
the deformation retraction of the previous section in these coordinates.
We also briefly describe the Kempf-Ness sets for some of these examples.

Consider a generating set of $\mathbf{G}$-invariant polynomials in
$\C[\mathfrak{R}_{r}(\mathbf{G})]^{\bG}$. Because these polynomials
distinguish orbits of the $\mathbf{G}$-action, this defines an embedding,
denoted $\phi$, of the $\mathbf{G}$-character variety $\mathfrak{X}_{r}(\mathbf{G})$
in a vector space $V:=\mathbb{C}^{N}$, given by sending an orbit
to all the values it takes on this generating set. The embedding realizes
$\mathfrak{X}_{r}(\mathbf{G})$ as a complex affine subvariety in
$V$, and the Euclidean topology mentioned above coincides with the
subspace topology induced from $V$ (see \cite[Section 2.3.3]{Florentino-Lawton:2009}).
Since $G\subset\mathbf{G}$ and $\mathfrak{R}_{r}(G)\subset\mathfrak{R}_{r}(\mathbf{G})$
we can try to use the generating set of invariants to relate $\mathfrak{X}_{r}(G)$
with the real points of $\mathfrak{X}_{r}(\mathbf{G})$, \[
\mathfrak{X}_{r}(\mathbf{G})(\mathbb{R})=\mathfrak{X}_{r}(\mathbf{G})\cap V(\mathbb{R}),\]
as follows. Since \[
\C[\mathfrak{R}_{r}(\mathbf{G})]^{\bG}\cong\R[\mathfrak{R}_{r}(\bG(\R))]^{\bG(\R)}\otimes_{\R}\C,\]
there exists a generating set for $\R[\mathfrak{R}_{r}(\bG(\R))]^{\bG(\R)}$,
which equals $\R[\mathfrak{R}_{r}(\bG(\R))]^{G}$ by density of $G$,
that extends (by scalars) to one for $\C[\mathfrak{R}_{r}(\mathbf{G})]^{\bG}$.
Thus, with respect to such a generating set, the real points $\X_{r}(\bG)(\R)$
are a well-defined real algebraic subset of $V(\R)$.

Denote by $f_{G}:\mathfrak{R}_{r}(G)\to V$ the composition of natural
maps\[
\mathfrak{R}_{r}(G)\subset\mathfrak{R}_{r}(\bG(\R))\subset\mathfrak{R}_{r}(\mathbf{G})\to\mathfrak{X}_{r}(\mathbf{G})\stackrel{\phi}{\to}V.\]
By $G$-invariance of $f_{G}$, this defines a map $p_{G}:\mathfrak{X}_{r}(G)\to V$
whose image lies in $V(\mathbb{R})$. Now, Proposition 6.8
in \cite{rich-slod:1990} states that the image $p_{G}(\mathfrak{X}_{r}(G))$,
is a closed subset in $\mathfrak{X}_{r}(\mathbf{G})(\mathbb{R})\subset V(\mathbb{R})$.
Thus, we have shown the following proposition.
\begin{prop}
\label{realpoints}Let $G$ be a fixed real form of a complex reductive
algebraic group $\mathbf{G}$. The set of real points $\mathfrak{X}_{r}(\mathbf{G})(\R)$
contains $p_{G}(\mathfrak{X}_{r}(G))$ as a closed subset.
Therefore, $\bigcup_{G}\, p_{G}(\mathfrak{X}_{r}(G))\subset\mathfrak{X}_{r}(\mathbf{G})(\R)$,
where the union is over all $G$ which are real forms of $\mathbf{G}$.
\end{prop}

The map $p_{G}$ is neither surjective nor injective
in general, and we will see below explicit examples that illustrate this situation.
Note, however, that for any real form $G$ of $\mathbf{G}$,
the map $p_{G}:\mathfrak{X}_{r}(G)\to\mathfrak{X}_{r}(\mathbf{G})(\mathbb{R})$
is always finite, as shown in \cite{rich-slod:1990}, Lemma 8.2.

\begin{rem}
The image of $p_{G}$ in Proposition \ref{realpoints} depends on the
given embedding $\phi$. For example, both the groups $\SO(2,\C)=\{x^{2}+y^{2}=1\}$
and $\GL(1,\C)=\{xy=1\}$ are isomorphic to $\C^{*}$ but the real
points for the first is $S^{1}$ and the real points for the second
is $\R^{*}$ (disconnected). So $\X_{r}(\bG)$ can have different
sets of real points, depending on the algebraic structure defined
by $\phi$.

Another example is $G=\SU(3)$ and $\mathbf{G}=\SL(3,\C)$. With respect to the
trace coordinates, $\SL(3,\C)\quot\SL(3,\C)$ can be identified with $\C^{2}$.
Its real points are $\R^{2}$ but the traces of $\SU(3)$ are not
real in general. So the image is not contained in the real locus with
respect to the trace coordinates. This means that, in each case, we should fix
a generating set having certain properties which avoids these issues.
\end{rem}

\label{sec:examples}

We now describe some particularly simple examples, where one can check
directly and explicitly the strong deformation retraction from $\mathfrak{X}_{r}(G)=G^{r}\quot G$
onto $K^{r}/K$, using the trace coordinates for the complex character
variety $\mathfrak{X}_{r}(\mathbf{G})$, and also by describing their
Kempf-Ness sets.

\subsection{$K=\SO(2)$}

Since the special orthogonal group $\SO(2)$ is Abelian, the conjugation
action is trivial. As $\SO(2)$ is isomorphic to the circle group,
$S^{1}$, it follows that
\begin{equation}
\mathfrak{X}_{r}(\SO(2))\cong (S^{1})^r.\label{eq:SO(2)}
\end{equation}

The maximal compact subgroup of $\SO(2,\C)\cong \C^*$ is $\SO(2)$, and it is clear that \[
\mathfrak{X}_{r}(\SO(2,\C))\cong (\C^*)^r\] deformation retracts to $(S^{1})^r$.

Of course, in general, for any Abelian group $G$, we have $\mathfrak{X}_{r}(G)=G^r$, and the deformation
retraction to $K^r$ is given, componentwise, by the polar decomposition.

\subsection{$G=\SL(2,\R)$, $r=1$}

The group $\SO(2)$ is also a maximal compact subgroup of $\SL(2,\R)$.
Hence, from \eqref{eq:SO(2)} and Theorem \ref{thm:maintheorem},
one concludes that $\mathfrak{X}_{1}(\SL(2,\R))$ also retracts onto
$S^{1}$. Let us also see this directly. $\mathfrak{X}_{1}(\SL(2,\R))$
is the space of closed orbits under the conjugation action. These
closed orbits correspond to diagonalizable matrices over $\C$. When
a matrix in $\SL(2,\R)$ is diagonalizable over $\R$, it corresponds
to a point in $\mathfrak{X}_{1}(\SL(2,\R))$ determined
by a matrix of the form $\diag(\lambda,\lambda^{-1})$ for some $\lambda\in\R^{*}$.
Since the diagonal matrices $\diag(\lambda,\lambda^{-1})$ and $\diag(\lambda^{-1},\lambda)$
are conjugated in $\SL(2,\R)$, we can suppose $\lambda\geq\lambda^{-1}$
(with equality exactly when $|\lambda|=1$) and thus the elements
of $\mathfrak{X}_{1}(\SL(2,\R))$ corresponding to these kind of matrices
are parametrized by the space \begin{equation}
\mathcal{D}_{\R}=\left\{ \diag(\lambda,\lambda^{-1})\,|\,\lambda\in\R\setminus(-1,1)\right\} .\label{eq:diagonalizable over R}\end{equation}
Similarly the space of matrices in $\SL(2,\R)$ diagonalizable over
$\C\setminus\R$ is parametrized by \[
\mathcal{D}_{\C}=\left\{ \diag(z,z^{-1}))\,|\, z\in\C^{*},\ z+z^{-1}\in\R\right\} .\]
Notice that we impose the condition of real trace, since the trace
is conjugation invariant. Now, for $z\in \C\setminus \R$, the condition $z+z^{-1}\in\R$ is equivalent
to $|z|=1$, so these matrices are in fact in $\SO(2,\C)$, hence
the corresponding ones in $\SL(2,\R)$ belong to $\SO(2)$, and have the
form \[
A_{\theta}=\left(\begin{array}{cc}
\cos\theta & -\sin\theta\\
\sin\theta & \cos\theta\end{array}\right)\]
 with $0\leq\theta<2\pi$. Now, the only possible $\SL(2,\C)$-conjugated matrices
of this type are $A_{\theta}$ and $A_{-\theta}$, and it is easily
seen that they are not conjugate in $\SL(2,\R)$. So, for each $\theta$, we
have a representative of a class in $\mathfrak{X}_{1}(\SL(2,\R))$.
Hence, from this and from \eqref{eq:diagonalizable over R}, we have a homeomorphism
\begin{equation}
\mathfrak{X}_{1}(\SL(2,\R))\cong\R\setminus(-1,1)\cup\{z\in\C\setminus\R\,|\,|z|=1\}.\label{eq:SL(2,R)}\end{equation}
 From \eqref{eq:SO(2)} and \eqref{eq:SL(2,R)} we see here directly
an example of our main theorem (Theorem \ref{thm:maintheorem}).

Using Proposition \ref{pro:r1-KN}, we can also obtain the same
space considering the Kempf-Ness quotient. For $G=\SL(2,\mathbb{R})$,
one can compute directly that the set of normal matrices is a union
of two closed sets, $\mathcal{KN}_{G}=Y_{1}\cup Y_{2}$, with:\[
Y_{1}=\left\{ \left(\begin{array}{cc}
\alpha & \gamma\\
\gamma & \beta\end{array}\right)\in \SL(2,\mathbb{R}):\ \alpha, \beta, \gamma \in \mathbb{R} \right\}, \]
and
\[
Y_{2}=\left\{A_{\theta}:\ \theta\in\mathbb{R}\right\} =\SO(2).\]
These correspond precisely to the $\SL(2,\mathbb{R})$ matrices that
are $\mathbb{R}$-diagonalizable or not, and they are distinguished
by the absolute value of their trace being greater or less than 2,
respectively (it is easy to show directly that the equation $\alpha\beta=1+\gamma^{2}$
for $\alpha,\beta,\gamma$ real implies $|\alpha+\beta|\geq2$). Now,
for an element $\mathcal{KN}_{G}=Y_{1}\cup Y_{2}$, besides the trace,
we have an extra invariant for the action of $\SO(2)$ (obviously
$Y_{2}$ is invariant under $\SO(2)$) which is the Pfaffian, defined
by (see \cite{As})
\[
\Pf(A)=c-b,\quad\mbox{for }A=\left(\begin{array}{cc}
a & b\\
c & d\end{array}\right).\]
We have $\Pf(B)=0$ for any $B\in Y_{1}$ and $\Pf(A_{\theta})=2 \sin\theta$.
So, the picture in Figure \ref{realrank1} is indeed a precise description of the
embedding $\mathcal{KN}_{G}/K\hookrightarrow\mathbb{R}^{2}$ under
the map $A\mapsto \frac12(\tr(A),\,\Pf(A))$,

\begin{figure}[ht]
\begin{center}
\includegraphics[scale=.2]{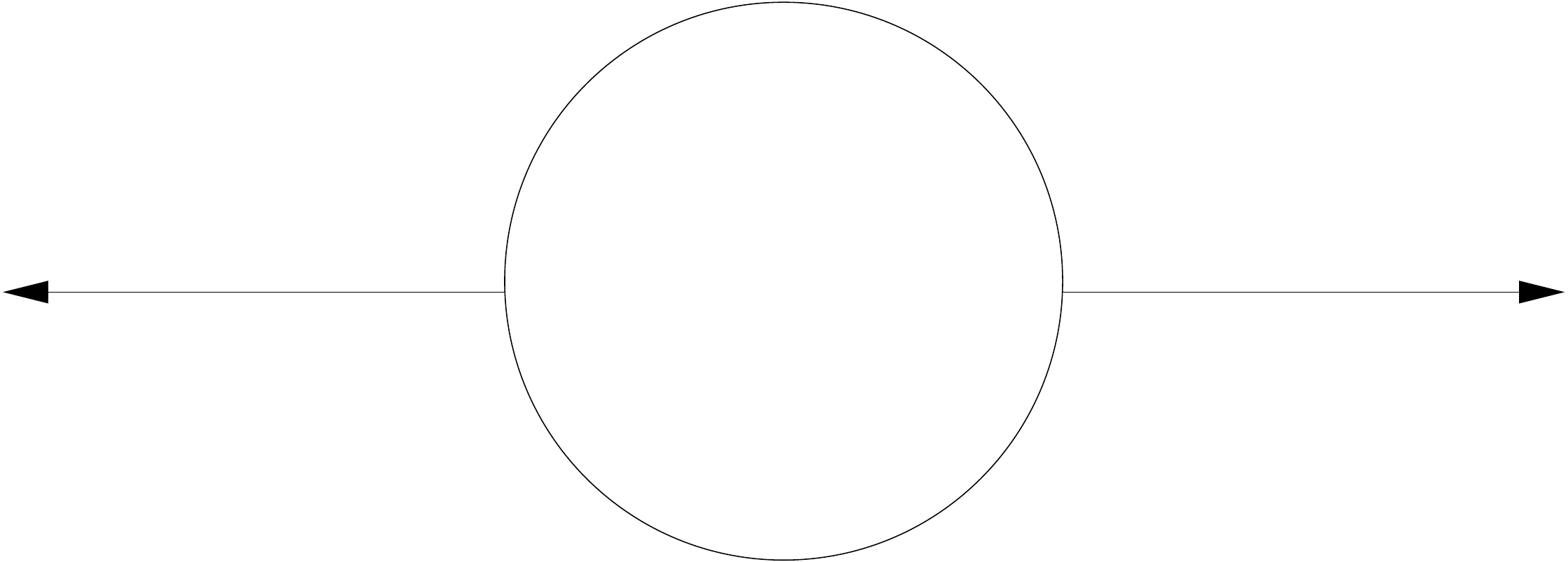}
\caption{$\mathfrak{X}_1(\SL(2,\R))\cong\R\setminus(-1,1)\cup\{z\in\C\setminus\R\,|\,|z|=1\}$}\label{realrank1}
\end{center}
\end{figure}

Finally, we can compare the geometry of $\mathfrak{X}_{1}(\SL(2,\C))=\SL(2,\C)\quot\SL(2,\C)$
and that of $\mathfrak{X}_{1}(\SL(2,\mathbb{R}))$, in trace coordinates.
In these coordinates, $\mathfrak{X}_{1}(\SL(2,\C))$
is $\C$ and thus its real points form $\R$. However, $\SL(2,\R)\quot\SL(2,\R)$
is $(-\infty,-1]\cup[1,\infty)\cup S^{1}$ where $S^{1}$ is a circle
centered at 0 of radius 1. However, after projecting to the real locus, the north and
south hemispheres in $S^{1}$ are identified (since they are conjugated
over $\C$ but not over $\R$). The image is then $\R$ and this is an example where the projection $p_G$ is not injective.

On the other
hand, considering $\SU(2)$, the quotient $\SU(2)/\SU(2)$ is $[-2,2]$
which projects to the same (since the projection is injective for
maximal compact subgroups). So we see that the projection to the real locus is not always surjective.

\subsection{$K=\SO(2)$, $r=2$, $G=\SL(2,\R)$}
Addressing the $\SL(2,\R)$ case with $r=2$ amounts to describing two real unimodular matrices up to $\SL(2,\R)$-conjugation. Generically, such a pair will correspond to an irreducible representation. The non-generic case is when $A_{1},A_{2}\in\SL(2,\R)$ are in the same torus; in particular, in this degenerate case they commute.

Consider first the complex invariants in $\SL(2,\C)^{\times2}$. By
Fricke-Vogt (see \cite{Goldman} for a nice exposition) we have an isomorphism\[
\mathfrak{X}_{2}(\SL(2,\C))=\SL(2,\C)^{\times2}\quot\SL(2,\C)\cong\mathbb{C}^{3},\]
explicitly given by $[(A_{1},A_{2})]\to(\mbox{tr}(A_{1}),\ \mbox{tr}(A_{2}),\ \mbox{tr}(A_{1}A_{2}))$.  Let $t_{1}=\mbox{tr}(A_{1})$, $t_{2}=\mbox{tr}(A_{2})$ and $t_{3}=\mbox{tr}(A_{1}A_{2})$.  Then $$\kappa(t_1,t_2,t_3):=\mathrm{tr}(A_1A_2A_1^{-1}A_2^{-1})=t_1^2+t_2^2+t_3^2-t_1t_2t_3-2.$$  Since commuting pairs $(A_1, A_2)$ have trivial commutator, the reducible locus is contained in $\kappa^{-1}(2)$.  The converse also holds (see \cite{CS}).

Thus, the $\R$-points of $\X_2(\SL(2,\C))$, here denoted by $\X_2(\SL(2,\C))(\R)$, form $\R^3$.  The reducible locus in $\X_2(\SL(2,\C))(\R)$ is therefore $\kappa^{-1}(2)\cap \R^3$ (see Figure \ref{reducibles}).

\begin{figure}[ht]
\begin{center}
\includegraphics[scale=.4]{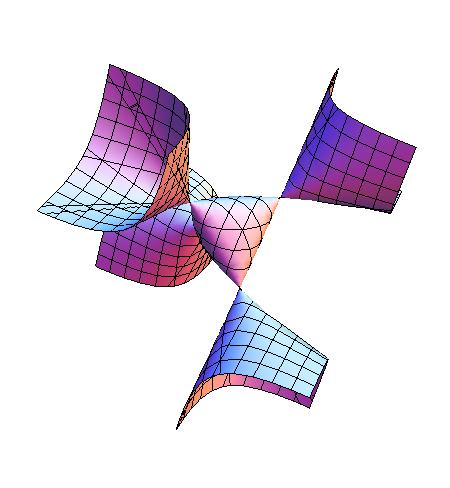}
\caption{Reducible Locus in $\X_2(\SL(2,\C))(\R)$}\label{reducibles}
\end{center}
\end{figure}

Suppose both $A_{1},A_{2}$ have eigenvalues of unit norm and commute.  So, up to conjugation in $\SL(2,\C)$, they are of the form \[
A_{1}=\left(\begin{array}{cc}
\cos\alpha & -\sin\alpha\\
\sin\alpha & \cos\alpha\end{array}\right),\quad A_{2}=\left(\begin{array}{cc}
\cos\beta & -\sin\beta\\
\sin\beta & \cos\beta\end{array}\right),\]
and thus $t_{1}  =  2\cos\alpha$, $t_{2}  =  2\cos\beta$, $t_{3}  =  2\cos(\alpha+\beta)$.  Note that in unitary coordinates these matrices take the form $\mathrm{diag}(e^{i\alpha},e^{-i\alpha})$ and $\mathrm{diag}(e^{i\beta},e^{-i\beta})$.  Putting these values in $\kappa(t_1,t_2,t_3)$ precisely determines the boundary of the solid closed 3-ball $\overline{B}^3\cong\X_2(\SU(2))$ depicted in Figure \ref{pillow} (see \cite[Lemma 6.3 (ii)]{Florentino-Lawton:2009}).

So the reducible locus is homeomorphic to $S^2$ and is given by those representations that, up to conjugation in $\SL(2,\C)$, are in $\SO(2)=\SL(2,\R)\cap\SU(2)$; this fact was first shown in \cite{Bratholdt-Cooper:2002}.

\begin{figure}[ht]
\begin{center}
\includegraphics[scale=.3]{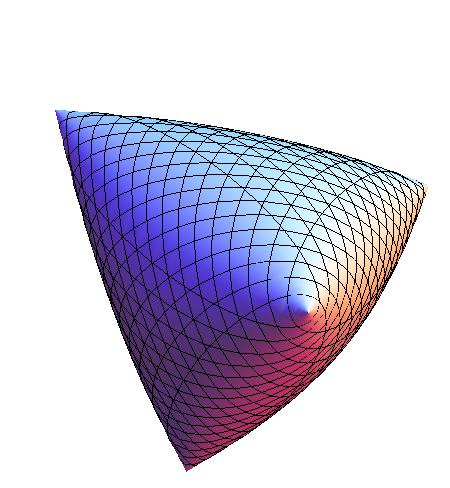}
\caption{$\mathfrak{X}_2(\SU(2))$}\label{pillow}
\end{center}
\end{figure}

Therefore, the four disjoint planes in Figure \ref{reducibles} correspond to pairs $$\left(\left(\begin{array}{cc}
\lambda & 0\\0 & 1/\lambda\end{array}\right), \left(\begin{array}{cc}\mu & 0\\0 & 1/\mu\end{array}\right)\right)$$ where $\lambda, \mu\in \R^*$.

As shown in \cite{ms} (p.458, Prop.III.1.1), every point in $\X_2(\SL(2,\C))(\R)\cong \R^3$ corresponds to either a $\SU(2)$-representation or a $\SL(2,\R)$-representation (this is a case where the union in Proposition \ref{realpoints} gives equality).  A point corresponds to a unitary representation if and only if $-2\leq t_1,t_2,t_3\leq 2$ and $\kappa(t_1,t_2,t_3)\leq 2$; as in Figure \ref{pillow}.  Otherwise, the representation is in $\SL(2,\R)$.

Figure \ref{region} gives a picture of $\X_2(\SL(2,\C))(\R)$, restricted to $\kappa(x,y,z)\leq 5$; from this, the deformation retraction to the boundary of the solid ball $\X_2(\SU(2))$ can be seen.

\begin{figure}[ht]
\begin{center}
\includegraphics[scale=.4]{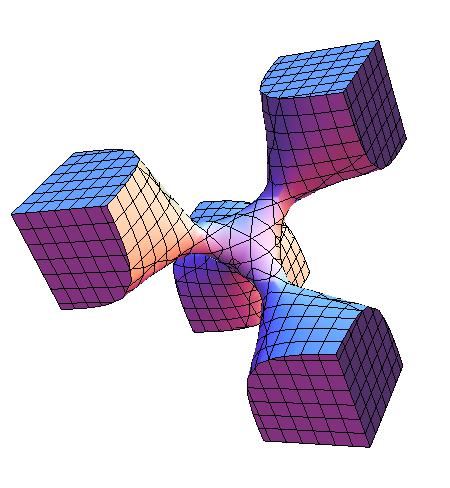}
\caption{Region $\kappa\leq 5$ in $\mathfrak{X}_2(\SL(2,\C))(\R)$}\label{region}
\end{center}
\end{figure}

Having described the $\R$-points of $\X_2(\SL(2,\C))$, we now can describe $\X_2(\SL(2,\R))$.  Since we are considering the closed $\SL(2,\R)$-orbits, we have three sets to consider.  First, there are the representations that are irreducible over $\SL(2,\C)$; these are called {\it absolutely irreducible}, or $\C$-irreducible.  Secondly, we have the representations that are irreducible over $\R$ but are reducible over $\C$; called $\C$-reducible or $\R$-irreducible.  Third, we have the representations that are reducible over $\R$; called $\R$-reducible.

Notice that conjugating by $\mathrm{diag}(i, -i)$ defines a $\Z/2\Z$-action on $\X_2(\SL(2,\R))$.  Since the characteristic polynomial is quadratic, any conjugation action from $\SL(2,\C)$ on $\X_2(\SL(2,\R))$ has to be equivalent to the action of $\mathrm{diag}(i, -i)$.  It is a free action on the irreducible locus (consisting of $\R$ and $\C$-irreducibles), and has its fixed locus exactly the $\R$-reducible representations.  The quotient of this action is exactly $\X_2(\SL(2,\C))(\R)-\X_2(\SU(2))^0=\R^3-B^3$ since $\mathrm{diag}(i, -i)$ preserves $\SL(2,\C)$-orbits.  Hence, the $\R$-reducible locus in $\X_2(\SL(2,\R))$ is exactly the four planes depicted in Figure \ref{reducibles}, and the irreducible locus in $\X_2(\SL(2,\R))$ is a $\Z/2\Z$-cover of the complement of those four planes in $\R^3-B^3$.

On the other hand, $\X_2(\SO(2))$ is a torus $S^1\times S^1$, and the sphere $S^2$ arises as the quotient of $S^1\times S^1$ by the $\Z/2\Z$-action described above (see \cite{Bratholdt-Cooper:2002}).

Lastly, the $\Z/2\Z$-action is equivariant with respect to the deformation retraction from $\X_2(\SL(2,\R))$ to $\X_2(\SO(2))$ since it is conjugation, which explains why we see the deformation retraction of $\R^3-B^3=\X_2(\SL(2,\C))(\R)-\X_2(\SU(2))^0\cong \X_2(\SL(2,\R))/(\Z/2\Z)$ onto $S^2=\partial \X_2(\SU(2))\cong \X_2(\SO(2))/(\Z/2\Z)$.

\begin{remark}
By \cite{Choi}, $\X_2(\SO(3))$ is a $(\Z/2\Z)^2$ quotient of $\X_2(\SU(2))$, and both relate to the moduli space of generalized spherical triangles.  This makes
$\X_2(\SO(3))$ an orbifold quotient of a solid $3-ball$ $($this also follows from \cite{Florentino-Lawton:2012}$)$.
\end{remark}

Now, let us consider the Kempf-Ness set for this case. From Proposition
\ref{pro:general-KN},\[
\mathcal{KN}_{G}^{Y}=\{(A_{1},A_{2})\in G^{2}:\ A_{1}^{*}A_{1}-A_{1}A_{1}^{*}+A_{2}^{*}A_{2}-A_{2}A_{2}^{*}=0\}.\]
Let us write the 2 matrices in convenient variables as\[
A_{i}=\left(\begin{array}{cc}
a_{i} & b_{i}\\
c_{i} & d_{i}\end{array}\right)=\left(\begin{array}{cc}
t_{i}+s_{i} & q_{i}-p_{i}\\
q_{i}+p_{i} & t_{i}-s_{i}\end{array}\right),\quad i=1,2.\]
So, the new variables are: $t_{i}=\frac{a_{i}+d_{i}}{2}$, $s_{i}=\frac{a_{i}-d_{i}}{2}$,
$q_{i}=\frac{c_{i}+b_{i}}{2}$ and $p_{i}=\frac{c_{i}-b_{i}}{2}$,
$i=1,2$. In particular, note that the traces and Pfaffians are $\tr(A_{i})=2t_{i}$
and $\Pf(A_{i})=2p_{i}$, respectively. We can describe inside $\mathbb{R}^{8}$
with coordinates $(a_{1},\cdots,d_{2})$ as the closed algebraic set:\[
\mathcal{KN}_{G}^{Y}=\left\{ (a_{1},\cdots,d_{2})\in\mathbb{R}^{8}:\ \begin{array}{c}
a_{1}d_{1}-b_{1}c_{1}=a_{2}d_{2}-b_{2}c_{2}=1,\\
\sum_{i=1}^{2}c_{i}^{2}-b_{i}^{2}=\sum_{i=1}^{2}(a_{i}-d_{i})(c_{i}-b_{i})=0\end{array}\right\} ,\]
or equivalently,\[
\mathcal{KN}_{G}^{Y}=\left\{ (t_{1},p_{1},s_{1},q_{1},t_{2},p_{2},s_{2},q_{2})\in\mathbb{R}^{8}:\ \begin{array}{c}
t_{1}^{2}+p_{1}^{2}-q_{1}^{2}-s_{1}^{2}=t_{2}^{2}+p_{2}^{2}-q_{2}^{2}-s_{2}^{2}=1,\\
q_{1}p_{1}+q_{2}p_{2}=s_{1}p_{1}+s_{2}p_{2}=0\end{array}\right\} ,\]
Now, $\mathcal{KN}_{G}^{Y}$ is invariant under $K=\SO(2)$, although
it is not immediately apparent. From \cite{As} we know that, for the
$\SO(2)$ simultaneous conjugation action on two $2\times2$ real matrices,
there are 8 invariants: the five traces $\tr A_{1}$, $\tr A_{2}$,
$\tr A_{1}A_{2}$, $\tr A_{1}^{2}$, $\tr A_{2}^{2}$, and the three
Pfaffians $\Pf A_{1}$, $\Pf A_{2}$ and $\Pf A_{1}A_{2}$. Obviously,
the traces of $A_{1}^{2}$ and $A_{2}^{2}$ are not important here
since for unimodular matrices $A$, $\tr A^{2}=(\tr A)^{2}-2$. So the 6
remaining invariants give rise to an embedding
\begin{eqnarray*}
\psi:\mathcal{KN}_{G}^{Y} /K& \to & \mathbb{R}^{6}\\{}
[(A_{1},A_{2})] & \mapsto & (t_{1},t_{2},t_{3},p_{1},p_{2},p_{3})
\end{eqnarray*}
with $t_{i}:=\frac{1}{2}\tr A_{i}$, $p_{i}:=\frac{1}{2}\Pf A_{i}$,
for $i=1,2$, and
\begin{eqnarray*}
t_{3} &\hspace{-.2cm} := & \hspace{-.2cm}\frac{1}{2}\tr(A_{1}A_{2})=t_{1}t_{2}-p_{1}p_{2}+s_{1}s_{2}+q_{1}q_{2}\\
p_{3} &\hspace{-.2cm} := &\hspace{-.2cm} \frac{1}{2}\Pf(A_{1}A_{2})=p_{1}t_{2}+t_{1}p_{2}+q_{1}s_{2}-s_{1}q_{2}.
\end{eqnarray*}
For $i=1,2$, let $\Delta_i=1-t_i^2-p_i^2$. Then the closure of the image is defined by the four equations:
\begin{eqnarray*}
p_{1}(t_{2}p_{1}+t_{1}p_{2}-p_{3}) & = & 0\\
p_{2}(t_{2}p_{1}+t_{1}p_{2}-p_{3}) & = & 0\\
p_{1}^{2}\Delta_{1}-p_{2}^{2}\Delta_2 & = & 0\\
p_{2}^{2}(\Delta_{1}(p_{1}^{2}-t_{1}^{2})-\Delta_{2}(p_{2}^{2}-t_{2}^{2})) & = & p_{3}\Delta_{1}(t_{1}p_{2}-t_{2}p_{1}),
\end{eqnarray*}
as can be obtained using a computer algebra system. So $\mathcal{KN}_{G}^{Y}/K$ is a semialgebraic set whose closure,
in these natural coordinates, is an algebraic set of degree 6 in $\mathbb{R}^{6}$,
and it can be checked that it has indeed dimension 3, as expected.

\subsection{$K=\SO(2)$, $r\geq 3$, $G=\SL(2,\R)$}
In this subsection we say a few words about the $r=3$ case.  The complex moduli space $\X_3(\SL(2,\C))$ is a branched double cover of $\C^6$ (it is a hyper-surface in $\C^7$).  Given a triple $(A_1,A_2,A_3)\in \hom(F_3,\SL(2,\C))$, the seven parameters determining its orbit closure are $t_i=\tr(A_i),t_k=\tr(A_iA_j),t_7=\tr(A_1A_2A_3)$, $1\leq i\not=j\leq 3$ and $4\leq k\leq 6$.  See \cite{Goldman} for details. From \cite{Florentino-Lawton:2009} and \cite{Bratholdt-Cooper:2002} it follows that $\X_3(\SL(2,\C))$ is homotopic to a 6-sphere $\X_3(\SU(2))\cong S^6$.

\begin{remark}
Fixing the values of the four parameters $\tr(A_i),\tr(A_1A_2A_3)$ defines relative character varieties since these are the four boundary coordinates for a 4-holed sphere; likewise, in the $r=2$ case fixing the boundary of a 1-holed torus is equivalent to fixing the value of $\kappa$.  The topology of the $\R$-points of relative character varieties for $r=2$ and $r=3$ have been explored in \cite{BG}; and some of their pictures relate to ours given that $\kappa$ arises naturally in both contexts.
\end{remark}

The defining equation for $\X_3(\SL(2,\C))$ is given by $$R=t_1^2-t_2 t_4 t_1-t_3 t_5 t_1+t_2 t_3 t_7 t_1-t_6 t_7t_1+t_2^2+t_3^2+t_4^2+t_5^2+t_6^2+t_7^2-t_2 t_3 t_6+t_4 t_5 t_6-t_3 t_4 t_7-t_2 t_5t_7-4.$$  In \cite{Florentino-Lawton:2012}, it is shown that the reducible locus is exactly the singular locus, and thus the Jacobian ideal $\mathfrak{J}$, generated by the seven partial derivatives of $R$, defines the reducible locus explicitly as a sub-variety.  Thus:  $\mathfrak{J}=\langle \partial R/\partial t_i\ |\ 1\leq i\leq 7\rangle=\langle 2 t_1-t_2 t_4-t_3 t_5+t_2 t_3 t_7-t_6 t_7,2 t_2-t_1 t_4-t_3 t_6+t_1 t_3 t_7-t_5
   t_7,2 t_3-t_1 t_5-t_2 t_6+t_1 t_2 t_7-t_4 t_7,-t_1 t_2+2 t_4+t_5 t_6-t_3 t_7,-t_1
   t_3+2 t_5+t_4 t_6-t_2 t_7,-t_2 t_3+t_4 t_5+2 t_6-t_1 t_7,t_1 t_2 t_3-t_4 t_3-t_2
   t_5-t_1 t_6+2 t_7\rangle.$

Using a Groebner basis algorithm, $\mathfrak{J}$ is equivalent to the ideal $\langle t_1^4-t_2 t_5 t_6 t_1^3-2 t_2^2 t_1^2+t_2^2 t_5^2 t_1^2-2 t_5^2 t_1^2+t_2^2 t_6^2 t_1^2+t_5^2 t_6^2 t_1^2-2 t_6^2 t_1^2-t_2 t_5 t_6^3 t_1-t_2 t_5^3 t_6 t_1-t_2^3 t_5 t_6 t_1+8 t_2 t_5 t_6 t_1+t_2^4+t_5^4+t_6^4-2 t_2^2 t_5^2-2 t_2^2 t_6^2+t_2^2 t_5^2 t_6^2-2 t_5^2 t_6^2\rangle$.  Thus, the reducible locus is isomorphic to a hypersurface in $\C^4$.  If all coordinates are restricted to $[-2,2]$, we are in $\X_3(\SU(2))$ since the $r=1$ case implies that $t_1, t_2, t_3$ are in $[-2,2]$ if and only if $A_1,A_2,A_3$ are $\SL(2,\C)$-conjugate to elements in $\SU(2)$, and that forces all the other coordinates to take values in $[-2,2]$ as well.

\begin{figure}[ht]
\begin{center}
\includegraphics[scale=.4]{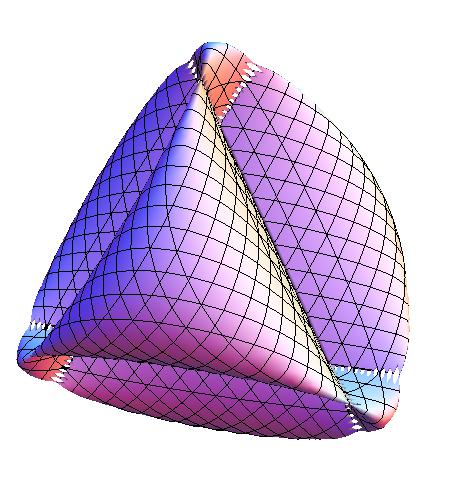}
\caption{A level set in the hypersurface $\mathfrak{X}_{\Z^3}(\SU(2))$}\label{levelset}
\end{center}
\end{figure}

As described in \cite{Florentino-Lawton:2013b}, the reducible locus in $\X_3(\SU(2))$ is homeomorphic to \linebreak$\X_{\Z^3}(\SU(2)):=\hom(\Z^3,\SU(2))/\SU(2)\cong (S^1)^3/(\Z/2\Z)$, and thus it is a 3-dimensional orbifold with 8 isolated singularities.  Neighborhoods around singularities will look like real cones over $\R P^2$; and thus are not locally Euclidean at those points (see p. 475 in \cite{ACG}).

Similar to the $r=2$ case, the orbifold $\X_{\Z^3}(\SU(2))$ is the quotient of $\X_3(\SO(2))\cong (S^1)^3$ by the $\Z/2\Z$-action defined from conjugation by $\diag(i,-i)$.  And as before, this action is equivariant with respect to the deformation retraction of $\X_3(\SL(2,\R))$ onto $\X_3(\SO(2))$, and it is a double cover over the absolutely irreducible representations and fixes the $\R$-reducible representations.

In fact, this situation is completely general.  $\X_r(\SL(2,\R))$ decomposes into three sets:  (1) the absolutely irreducible locus which double covers the irreducible locus in the $\X_r(\SL(2,\C))(\R)$; (2) the $\R$-irreducible locus isomorphic to $(S^1)^r=\X_r(\SO(2))$ branch double covers the orbifold $\X_{\Z^r}(\SU(2))\cong (S^1)^r/(\Z/2\Z)$ having $2^r$ discrete fixed points from the central representation, making orbifold singularities with neighborhoods isomorphic to real cones over $\R P^{r-1}$; and (3) the $\R$-reducible locus which is isomorphic to that same locus in $\X_r(\SL(2,\C))(\R)$ intersecting $\X_{\Z^r}(\SU(2))$ at the central representations and homeomorphic to $((-\infty, -1]\cup [1,\infty))^r$.  And the deformation retraction we establish in this paper from $\X_r(\SL(2,\R))$ to $\X_r(\SO(2))$ is equivariant with respect to the $\Z/2\Z$-action of $\mathrm{diag}(i,-i)$ and therefore determines a deformation retraction on the level of $\R$-points from $\X_r(\SL(2,\C))(\R)-\X_r(\SU(2))^0$ onto $\X_{\Z^r}(\SU(2))$. For an explicit characterization of the $\mathbb{C}$-reducible
locus, in terms of traces of minimal words, see \cite{Flo}.

\section*{Acknowledgments}
We thank Tom Baird for many conversations about equivariant cohomology and Theorem \ref{trivialaction}.  We also thank the referee for a careful reading, and for many suggestions leading to relevant improvements in exposition.

\end{document}